\documentclass{amsart}

\usepackage[T1]{fontenc}
\usepackage[utf8]{inputenc}

\usepackage{amsmath}
\usepackage{amssymb}

\newtheorem{theorem}{Theorem}[section]

\newtheorem{lemma}[theorem]{Lemma}
\newtheorem{proposition}[theorem]{Proposition}
\newtheorem{corollary}[theorem]{Corollary}

\begin{document}
\title{Odd Catalan numbers modulo $2^k$}
\author{Hsueh-Yung Lin}
\address{\'Ecole Normale Sup\'erieure de Lyon}
\date{\today}

\begin{abstract}
This article proves a conjecture by S.-C.~Liu and C.-C.~Yeh about Catalan numbers, which states that odd Catalan numbers can take exactly $k-1$ distinct values modulo~$2^k$, namely the values $C_{2^1-1},\ldots,\allowbreak C_{2^{k-1}-1}$.
\end{abstract}

\maketitle

\setcounter{section}{-1}

\section{Notation}

In this article we denote $C_n := \frac{(2n)!}{(n+1)!n!}$ the $n$-th Catalan number. We also define $(2n+1)!! := 1\times 3\times\cdots\times(2n+1)$. For $x$ an integer, $\nu_2(x)$ stands for the $2$-adic valuation of~$x$, i.e.\ $\nu_2(x)$ is the largest integer $a$ such that $2^a$ divides~$x$.

\section{Introduction}

The main result of this article is Theorem~\ref{thmp}, which proves a conjecture by S.-C.~Liu and C.-C.~Yeh about odd Catalan numbers~\cite{key-9}. To begin with, let us recall the characterization of odd Catalan numbers:
\begin{proposition}
A Catalan number $C_n$ is odd if and only if $n=2^a-1$ for some integer~$a$.
\end{proposition}
\noindent That result is easy, see e.g.~\cite{key-2}.

The main theorem we are going to prove is the following:

\begin{theorem} \label{thmp}
For all $k\geq2$, the numbers $C_{2^1-1}, C_{2^2-1},\ldots, C_{2^{k-1}-1} $ all are distinct modulo~$2^k$, and modulo~$2^k$ the sequence $(C_{2^n-1})_{n\geq 1}$ is constant from rank ${k-1}$ on.
\end{theorem}

Here are a few historical references about the values of the $C_n$ modulo~$2^k$. Deutsch and Sagan~\cite{key-1} first computed the $2$-adic valuations of the Catalan numbers. Next S.-P.~Eu, S.-C.~Liu and Y.-N.~Yeh~\cite{key-5} determined the modulo~$8$ values of the~$C_n$. Then S.-C.~Liu et C.-C.~Yeh determined the modulo~$64$ values of the~$C_n$ by extending the method of Eu, Liu and Yeh in~\cite{key-9}, in which they also stated Theorem~\ref{thmp} as a conjecture.

Our proof of Theorem~\ref{thmp} will be divided into three parts. In Section~\ref{S2} we will begin with the case $k=2$, which is the initialization step for a proof of Theorem~\ref{thmp} by induction. In Section~\ref{S3} we will prove that the numbers $C_{2^1-1}, C_{2^2-1},\ldots, C_{2^{k-1}-1}$ all are distinct modulo~$2^k$. Finally in Section~\ref{S4} we will prove that $C_{2^n-1} \equiv C_{2^{k-1}-1} \pmod{2^k}$ for all ${n\geq k-1}$.

\section{Odd Catalan numbers modulo~$4$}\label{S2}

In this section we prove that any odd Catalan number is congruent to $1$ modulo~$4$, which is Theorem~\ref{thmp} for ${k=2}$. Though this result can be found in~\cite{key-5}, I give a more ``elementary'' proof, in which I shall also make some computations which will be used again in the sequel.

Before starting, we state two identities:

\begin{lemma}\label{lemeg}
For any $a\geq 3$, the following identities hold: 
\begin{equation}\label{1}
(2^{a}-1)!!\equiv 1\pmod{2^{a}};
\end{equation} 
\begin{equation}\label{2}
(2^a-3)!!\equiv -1 \pmod{2^{a+1}}.
\end{equation} 
\end{lemma}

\begin{proof}
We are proving the two identities separately. In both cases we reason by induction on~$a$, both equalities being trivial for ${a=3}$. So, let $a \geq 4$ and suppose the result stands true for ${a-1}$. First we have
 \begin{eqnarray*}
(2^{a}-1)!! & = & 1\times 3\times\cdots\times(2^{{a-1}}-1)\times(2^{{a-1}}+1)\times\cdots\times(2^{a}-1)\\
 & \equiv & 1\times3\times\cdots\times(2^{{a-1}}-1)\times(-(2^{{a-1}}-1))\times\cdots\times(-1)\\
 & = & (1\times3\times\cdots\times(2^{{a-1}}-1))^{2}\times (-1)^{2^{a-2}} \pmod{2^{a}}.
 \end{eqnarray*}
Since, by the induction hypothesis, $1\times3\times\cdots\times(2^{a-1}-1)$ is equal to $1$ or ${2^{a-1}+1}$ modulo~$2^a$, we have $(1\times3\times\cdots\times(2^{{a-1}}-1))^2\equiv1\pmod{2^a}$ in both cases, from which the first identity follows.

For the second identity,
\[(2^{a}-3)!! = \prod_{k=1}^{2^{a-2}-1}(2k+1) \cdot \prod_{k=2^{a-2}}^{2^{a-1}-2}(2k+1).\]
Reversing the order of the indexes in the first product and translating the indexes in the second one, we get
 \begin{eqnarray*}
 (2^{a}-3)!!  & = & \prod_{k=0}^{2^{a-2}-2}(2^{a-1}-(2k+1)) \cdot \prod_{k=0}^{2^{a-2}-2}(2^{a-1}+(2k+1))\\ 
 & = & \prod_{k=0}^{2^{a-2}-2}[2^{2(a-1)}-(2k+1)^2]\\
 & \equiv & \prod_{k=0}^{2^{a-2}-2}[-(2k+1)^2] = -(2^{{a-1}}-3)!!^{2} \pmod{2^{{a+1}}}.
 \end{eqnarray*}
By the induction hypothesis, $(2^{{a-1}}-3)!!$ is equal to ${-1}$ or ${2^{a}-1}$ modulo $2^{a+1}$, and in either case the result follows.
\end{proof}

Now comes the main proposition of this section:

\begin{proposition}\label{thmk=2}
Fore all integer~$a$, $C_{2^a-1}\equiv 1 \pmod{4}$.
\end{proposition}

\begin{proof}
Put $n := 2^a-1$. We want to prove that $4\mid \frac{(2n)!}{n!(n+1)!}-1=\frac{(2n)!-n!(n+1)!}{n!(n+1)!}$. Let us denote $\omega := \nu_2[(2n)!]$. Since $C_n = \frac{(2n)!}{n!(n+1)!}$ is odd, one also has $\omega = \nu_2[n!(n+1)!]$. Then, proving that $4\mid \frac{(2n)!-n!(n+1)!}{n!(n+1)!}$ is equivalent to proving that $4\mid\frac{(2n)!}{2^{\omega}}-\frac{n!(n+1)!}{2^{\omega}}$. To do that, it suffices to show that $\frac{n!(n+1)!}{2^{\omega}} \equiv 1 \pmod4$ and $\frac{(2n)!}{2^{\omega}} \equiv 1 \pmod4$.

As $\omega = \nu_2[n!(n+1)!] =  \nu_2[(n!)^22^a] = a+2\nu_2(n!)$, one has $\nu_2(n!)=(\omega - a)/2$, thus $n!/2^{(\omega-a)/2}$ is an odd number by the very definition of valuation. That yields the first equality:
\[
\frac{n!(n+1)!}{2^{\omega}} = \frac{(n!)^{2}(n+1)}{2^{\omega}} = \frac{(n!)^{2}2^{a}}{2^{\omega}}
= \left(\frac{n!}{2^{(\omega-a)/2}}\right)^{2} \equiv 1 \pmod{4}.
\]

Concerning the equality $ \frac{(2n)!}{2^{\omega}}\equiv 1 \pmod4 $, it is easy to check for ${a\leq2}$; now we consider the case ${a\geq3}$, to which we can apply Lemma~\ref{lemeg}. For all $i\leq a$, put $\omega_i :=\nu_2[(2^{a-i+1}-1)!]$. For $i<a$, one has
\[
\frac{(2^{a-i+1}-1)!}{2^{\omega_i}} = \frac{(2^{a-i+1}-1)!! (\prod_{p=1}^{2^{a-i}-1}2p)}{2^{\omega_i}}=(2^{a-i+1}-1)!!\frac{(2^{a-i}-1)!}{2^{\omega_i+2^{a-i}-1}}.
\]
As the left-hand side of this equality is odd, so is its right-hand side, so that ${\omega_i+2^{a-i}-1}$ is actually the $2$-adic valuation of ${2^{a-i}-1}$. In the end, we have shown that
\[
\frac{(2^{a-i+1}-1)!}{2^{\omega_i}} = (2^{a-i+1}-1)!!\frac{(2^{a-i}-1)!}{2^{\omega_{i+1}}}.
\]
Morevoer, for $i=a$ it is immediate that $(2^{a-i+1}-1)!/2^{\omega_i} = 1$, whence
 \begin{eqnarray*} 
\frac{(2n)!}{2^{\omega}}&=& \frac{(2^{a+1}-2)!}{2^{\omega}} = \frac{1}{2^{a+1}-1}\cdot\frac{(2^{a+1}-1)!}{2^{\omega}}\\
&=& \frac{1}{2^{a+1}-1}\cdot(2^{a+1}-1)!!\cdot \frac{(2^{a}-1)!}{2^{\omega_{1}}}\\
&=& \frac{1}{2^{a+1}-1}\cdot(2^{a+1}-1)!!\cdot (2^{a}-1)!!\cdot  \frac{(2^{a-1}-1)!}{2^{\omega_{2}}}\\
&=& \cdots \rule{0pt}{4ex}\\ 
&=& \frac{1}{2^{a+1}-1}\cdot \prod_{k=1}^{a+1}(2^{k}-1)!!\\
&=& \frac{1}{2^{a+1}-1}\cdot(2^{a+1}-1)!!\cdot \prod_{k=1}^{a}(2^{k}-1)!!\\
&=& (2^{a+1}-3)!!\cdot \prod_{k=1}^{a}(2^{k}-1)!!.
 \end{eqnarray*}
But, modulo~$4$, one has $(2^{a+1}-3)!! \equiv -1$ by~\eqref{2} in Lemma~\ref{lemeg}, $(2^1 -1)!! \equiv 1$, $(2^2 -1)!! \equiv -1$ and $(2^k -1)!! \equiv 1$ for ${k\geq3}$ by~\eqref{1} in Lemma~\ref{lemeg}, whence $(2n)!/2^{\omega} {\equiv 1}$.
\end{proof}

Before ending this section, I highlight an intermediate result of the previous proof by stating it as a lemma:

\begin{lemma} \label{lemcalc} For $a \geq 0$, putting $\omega := \nu_2[(2^{a+1}-2)!]$,
\[ \frac{(2^{a+1}-2)!}{2^{\omega}} = (2^{a+1}-3)!!\cdot \prod_{k=1}^{a}(2^{k}-1)!! .\]
\end{lemma}

\section{Distinctness modulo $2^k$ of the $C_{2^1-1},\ldots,C_{2^{k-1}-1}$}\label{S3}

In this section we prove that for all $k\geq2$, the numbers $C_{2^1-1},\ldots,\allowbreak C_{2^{k-1}-1} $ are distinct modulo~$2^k$.
To begin with, we state a lemma which gives an equivalent formulation to the equality ``$C_{2^{m}-1}\equiv p \pmod{2^k}$''. This lemma will be used in Sections~\ref{S3} and~\ref{S4}.

\begin{lemma}\label{lemmod}
Let $k\geq 2$ and $m\geq 1$, then $C_{2^{m}-1}\equiv p \pmod {2^k}$ if and only if
 \[
(2^{m+1}-3)!!\equiv p\prod_{n=1}^{m}(2^{n}-1)!! \pmod{2^{k}}.
 \]
\end{lemma}

\begin{proof}
Denote $\omega := \nu_2[(2^{m+1}-2)!] = \nu_2[(2^{m})!(2^{m}-1)!]$ (recall that $C_{2^{m}-1}=\frac{(2^{m+1}-2)!}{(2^{m})!(2^{m}-1)!}$ is odd). Applying Lemma~\ref{lemcalc},
\begin{eqnarray*}
&& C_{2^m-1}\equiv p \pmod {2^k}\\
&\Leftrightarrow& 2^k\mid \frac{(2^{m+1}-2)!}{(2^m)!(2^m-1)!}-p\\
&\Leftrightarrow& 2^k\mid\frac{(2^{m+1}-2)!}{2^{\omega}}-\frac{p(2^m)!(2^m-1)!}{2^{\omega}} \\
&\Leftrightarrow& 2^{k}\mid (2^{m+1}-3)!!\prod_{n=1}^{m}(2^{n}-1)!!-p\left(\prod_{n=1}^{m}(2^{n}-1)!!\right)^2\\
&\Leftrightarrow& 2^{k}\mid \left(\prod_{n=1}^{m}(2^{n}-1)!!\right)\, \left((2^{m+1}-3)!!-p\prod_{n=1}^m(2^n-1)!!\right).
\end{eqnarray*}
But $\prod_{n=1}^{m}(2^{n}-1)!!$ is odd, so $C_{2^m-1}\equiv p \pmod {2^k}$ if and only if $2^k$ divides $({(2^{m+1}-3)!!}-p\prod_{n=1}^m(2^n-1)!!)$, which is our lemma.
\end{proof}

\begin{proposition} \label{propdist}
Let $k\geq2$ be an integer. For all $j \in \{1,\ldots,k-1\}$, $C_{2^j-1}\not\equiv C_{2^k-1} \pmod{2^{k+1}}$.
\end{proposition}

\begin{proof}
We prove this proposition by contradiction. Suppose there exists a $j\in\{1,\ldots,{k-1}\}$ such that $C_{2^j-1}\equiv C_{2^k-1} =: p\pmod{2^{k+1}}$. By Lemma~\ref{lemmod}, one would have
\[ p\prod_{n=1}^{j} (2^{n}-1)!! \equiv (2^{j+1}-3)!!   \pmod{2^{k+1}}, \]
and by Lemma~\ref{lemmod} and Fomula~\eqref{2} in Lemma~\ref{lemeg},
\[ p\prod_{n=1}^{k} (2^{n}-1)!! \equiv (2^{k+1}-3)!!  \equiv  -1   \pmod{2^{k+1}}. \]
As $j+2\leq k+1$, both equalities would remain true modulo~$2^{j+2}$. Thus one would have
\begin{eqnarray*}
-1 &\equiv &p\prod_{n=1}^{k} (2^{n}-1)!! \pmod{2^{j+2}}\\
&=&p \prod_{n=1}^{j} (2^{n}-1)!! \times \prod_{n=j+1}^{k} (2^{n}-1)!!\\
&\equiv &(2^{j+1}-3)!! \times \prod_{n=j+1}^k (2^n-1)!!\\
&=& (2^{j+1}-3)!!\cdot (2^{j+1}-1)!! \times \prod_{n=j+2}^{k} (2^{n}-1)!!\\
&\equiv& (2^{j+1}-3)!!\cdot (2^{j+1}-1)!!  \quad \textrm{(by~\eqref{1} in Lemma~\ref{lemeg})} \rule{0pt}{3ex}\\
&=&(2^{j+1}-3)!!^2 \cdot (2^{j+1}-1) \rule{0pt}{3ex}\\
&\equiv &2^{j+1}-1\pmod{2^{j+2}} \quad \textrm{(by~\eqref{2} in Lemma~\ref{lemeg})}, \rule{0pt}{3ex}
\end{eqnarray*}
which is absurd.
\end{proof}

Thanks to the previous proposition, we prove the first claim of Theorem~\ref{thmp}:

\begin{corollary} \label{coro}
For $k\geq2$, the numbers $C_{2^1-1}, C_{2^2-1}, \ldots, C_{2^{k-1}-1}$ all are distinct modulo~$2^k$.
\end{corollary}

\begin{proof}
The case $k=2$ is trivial. Let $k\geq2$ and suppose that, modulo~$2^k$, the numbers $C_{2^1-1}, C_{2^2-1},\ldots,C_{2^{k-1}-1}$ all are distinct, so that they are also distinct modulo~$2^{k+1}$. By Proposition~\ref{propdist}, $C_{2^{j}-1}\not\equiv C_{2^{k}-1}$ (mod $2^{k+1}$) for all $j\in\{1,\ldots,{k-1}\}$, so the numbers $C_{2^1-1}, C_{2^2-1},\ldots, C_{2^k-1} $ all are distinct modulo~$2^{k+1}$. The claim follows by induction.
\end{proof}

\section{Ultimate constancy of the sequence of the $C_{2^n-1}$ modulo~$2^k$}\label{S4}

To complete the proof of Theorem~\ref{thmp}, it remains to prove that the $C_{2^n-1}$ all are equal modulo~$2^k$ for $n\geq{k-1}$.

\begin{proposition}
Let $k\geq 2$, then for all $m\geq {k-1}$, $C_{2^m-1}\equiv C_{2^{k-1}-1}\pmod {2^k}$.
\end{proposition}

\begin{proof}
Denote $C_{2^{k-1}-1} =: p \pmod{2^k}$. We will show that $C_{2^m-1}\equiv p \pmod{2^k}$ for all $m\geq {k-1}$ by induction. Let $m\geq k$ be such that the previous equality stands true for~${m-1}$. By Lemma~\ref{lemmod}, it suffices to show that $(2^{m+1}-3)!!\equiv p  \prod_{n=1}^m(2^n-1)!! \pmod{2^k}$. To do this, we are going to show that $(2^{m+1}-3)!! \equiv (2^m-3)!! \pmod{2^k}$ and that $p\prod_{n=1}^m(2^n-1)!! \equiv (2^m-3)!! \pmod{2^k}$.

The first equality follows from the following computation:
\begin{eqnarray*}
&& (2^{m+1}-3)!!\\ 
&=& (2^{m}-3)!! \times (2^{m}-1)\times (2^{m}+1) \times\cdots\times (2\cdot 2^{m}-3)\\
&\equiv& (2^{m}-3)!! \cdot \left(1\times 3\times\cdots\times (2^k-1)\right)^{2^{m-k}}  \pmod{2^k}\\
&\equiv& (2^{m}-3)!! \quad \textrm{(by~\eqref{1} in Lemma~\ref{lemeg})}.
\end{eqnarray*}

To get the other equality, using again~\eqref{1} in Lemma~\ref{lemeg}, one has
\[ p\prod_{n=1}^{m}(2^{n}-1)!! = (2^{m}-1)!!\cdot p \prod_{n=1}^{m-1}(2^{n}-1)!! \equiv p\prod_{n=1}^{m-1}(2^{n}-1)!! \pmod{2^k}. \]
But by Lemma~\ref{lemmod}, the induction hypothesis means that $p\prod_{n=1}^{m-1}(2^{n}-1)!! \equiv {(2^m-3)!!} \pmod {2^k}$, whence the result.
\end{proof}

\section{Going further}

After the series of works on odd Catalan numbers modulo~$2^k$ this article belongs to, a natural question would be how many distinct \emph{even} Catalan numbers there are modulo~$2^k$ and how these numbers behave. An idea to do this would be to study the $C_n$ having some fixed $2$-adic valuation.

More generally, one could also wonder what happens for Catalan numbers modulo $p^k$ for prime~$p$, which is a question that mathematicians studying the arithmetic properties of Catalan numbers have been asking for a long time.

\section*{Acknowledgements}

The author thanks Pr P.~Shuie and Pr S.-C.~Liu for their mathematical advice, and R.~Peyre for helping to improve the writing of this article.

\bibliographystyle{amsplain}
\bibliography{article}
\end{document}